\documentclass[a4paper,10pt]{amsart}
\usepackage{tikz-cd, amssymb, amsthm, hyperref, breakurl}
\usepackage[nameinlink]{cleveref}
\theoremstyle{plain}
\newtheorem{theorem}{Theorem}
\newtheorem{lemma}[theorem]{Lemma}
\newtheorem{corollary}[theorem]{Corollary}

\title[Simple graphs are coreflective in groups under free groups]{The category of simple graphs is coreflective in the comma category of groups under the free group functor}
\author{Christian Frank}

\begin{document}
 
 \begin{abstract}
  We show that the comma category $(\mathcal{F}\downarrow\mathbf{Grp})$ of groups under the free group functor $\mathcal{F}: \mathbf{Set} \to \mathbf{Grp}$ contains the category $\mathbf{Gph}$ of simple graphs as a full coreflective subcategory.  More broadly, we generalize the embedding of topological spaces into Steven Vickers' category of topological systems to a simple technique for embedding certain categories into comma categories, then show as a straightforward application that simple graphs are coreflective in $(\mathcal{F}\downarrow\mathbf{Grp})$.
 \end{abstract}
 
 \maketitle
 
 \section{Introduction}
 
 In his 1989 text ``Topology Via Logic'' \cite{MR1002193}, Steven Vickers develops a category $\mathbf{TopSys}$ of topological systems which is noteworthy for containing both the category of topological spaces as a coreflective subcategory and the category of locales as a reflective subcategory.  In \cite{MR1474566}, Adamek and Pedicchio observe that $\mathbf{TopSys}^\mathrm{op}$ is equivalent to the comma category $(\mathbf{Frm}\downarrow\mathbf{CABA})$ of frames over complete atomic boolean algebras, and further show that $\mathbf{Top}^\mathrm{op}$ is equivalent to the regularly epireflective subcategory of monomorphisms.
 
 Rewritten, $\mathbf{Top}$ is the coreflective subcategory of epimorphisms in
 \[
  (\mathbf{Frm}\downarrow\mathbf{CABA})^\mathrm{op} = (\mathbf{CABA}^\mathrm{op}\downarrow\mathbf{Frm}^\mathrm{op}).
 \]
 As the category of complete atomic boolean algebras is opposite the category of sets via taking the powerset, a topological system can be interpreted as a localic map out of a powerset, ie as an object in the comma category $(\mathcal{P}\downarrow\mathbf{Loc})$ of locales under the powerset locale functor $\mathcal{P}: \mathbf{Set} \to \mathbf{Loc}$.
 
 If $2$ is the two-open locale, the hom functor $\hom_\mathbf{Loc}(2,-)$ assigns to each locale its set of locale points; this functor is right adjoint to $\mathcal{P}$.  By adjunction, the category of topological systems is also equivalent to $(\mathbf{Set}\downarrow\hom_\mathbf{Loc}(2,-))$.
 
 Note that $\mathbf{Top}$ fits right in the middle of the adjunction $\mathcal{P}\dashv\hom_\mathbf{Loc}(2,-)$.  If $\mathcal{D}_T: \mathbf{Set} \to \mathbf{Top}$ is the discrete space functor, $\mathcal{U}: \mathbf{Top} \to \mathbf{Set}$ the forgetful functor right adjoint to $\mathcal{D}_T$, and $\mathcal{L\dashv S}$ is the locale-spectrum adjunction between $\mathbf{Top}$ and $\mathbf{Loc}$ (see eg Chapter 2 of Picado and Pultr \cite{MR2868166} for details), then $\mathcal{P} = \mathcal{L\circ D}_T$ and $\hom_\mathbf{Loc}(2,-) = \mathcal{U\circ S}$.
 
 In this language, $\mathbf{TopSys} \cong (\mathcal{L\circ D}_T\downarrow\mathbf{Loc}) \cong (\mathbf{Set}\downarrow\mathcal{U\circ S})$.  Adjuncting into the category $\mathbf{Top}$, we have $\mathbf{TopSys} \cong (\mathcal{D}_T\downarrow\mathcal{S})$.  Here $\mathbf{Top}$ embeds naturally as a composition of unit with counit: the map $\eta^{\mathcal{L\dashv S}}\circ\varepsilon^{\mathcal{D}_T\dashv\mathcal{U}}: \mathcal{D}_T\circ\mathcal{U} \to \mathrm{id} \to \mathcal{S\circ L}$ gives for each space an object of the comma category $(\mathcal{D}_T\downarrow\mathcal{S})$.
 
 In this paper we investigate conditions for which categories in the center of two composable adjoint pairs may be embedded nicely into the appropriate comma categories through this process.
 
 \section{The unit-counit composite}\label{composite}
 
 Let $\mathbf{B},\mathbf{C},\mathbf{D}$ be complete and cocomplete categories.  Let
 \begin{align*}
  F: \mathbf{B}&\rightleftarrows\mathbf{C} :G,\\
  L: \mathbf{C}&\rightleftarrows\mathbf{D} :R,
 \end{align*}
 with $F \dashv G$ and $L \dashv R$, and define $\gamma: \mathbf{C} \to (F\downarrow R)$ by
 \begin{align*}
  \gamma(A) &: F(G(A)) \to R(L(A)),\\
  \gamma(A) &:= \eta^{L \dashv R}_A\circ\varepsilon^{F \dashv G}_A,\\
  \gamma(f) &:= (G(f)\downarrow L(f)).
 \end{align*}
 
 We ask what conditions will ensure that $\gamma$ embeds $\mathbf{C}$ as a reflective subcategory of the equivalent comma categories
 \[
  (\mathbf{B}\downarrow(G\circ R)) \cong (F\downarrow R) \cong ((L\circ F)\downarrow\mathbf{D}).
 \]
 
 We can say immediately that $\gamma$ is faithful as soon as either $G$ is or $L$ is.
 
 Fullness is more precarious.  For $A,B$ objects in $\mathbf{C}, f_\mathbf{B}\in\hom_\mathbf{B}(G(A),G(B))$, $f_\mathbf{D}\in\hom_\mathbf{D}(L(A),L(B))$, assume that we have a diagram in $\mathbf{C}$:
 \begin{equation}
  \begin{tikzcd}\label{fullness-initial}
   (F\circ G)(A) \arrow[swap]{d}{\gamma(A)} \arrow{r}{F(f_\mathbf{B})} &(F\circ G)(B) \arrow{d}{\gamma(B)}\\
   (R\circ L)(A) \arrow[swap]{r}{R(f_\mathbf{D})} &(R\circ L)(B)
  \end{tikzcd}
 \end{equation}
 
 We are interested in conditions that will guarantee that the pair $(f_\mathbf{B},f_\mathbf{D})$ is in the image of $\gamma$, that is, conditions under which the commutativity of (\ref{fullness-initial}) implies existence of $f\in\hom_\mathbf{C}(A,B)$ satisfying both $f_\mathbf{B} = G(f)$ and $f_\mathbf{D} = L(f)$.
 
 The above diagram expands to:
 \begin{equation}
  \begin{tikzcd}\label{fullness-unfilled}
   (F\circ G)(A) \arrow[swap]{d}{\varepsilon^{F\dashv G}} \arrow{r}{F(f_\mathbf{B})} &(F\circ G)(B) \arrow{d}{\varepsilon^{F\dashv G}}\\
   A \arrow[swap]{d}{\eta^{L\dashv R}} &B \arrow{d}{\eta^{L\dashv R}}\\
   (R\circ L)(A) \arrow[swap]{r}{R(f_\mathbf{D})} &(R\circ L)(B)
  \end{tikzcd}
 \end{equation}
 
 Suppose there is some $f'\in\hom_\mathbf{C}(A,B)$ filling in the diagram:
 \begin{equation}
  \begin{tikzcd}\label{fullness-filled}
   (F\circ G)(A) \arrow[swap]{d}{\varepsilon^{F\dashv G}} \arrow{r}{F(f_\mathbf{B})} &(F\circ G)(B) \arrow{d}{\varepsilon^{F\dashv G}}\\
   A \arrow[swap]{d}{\eta^{L\dashv R}} \arrow[dashed]{r}{f'} &B \arrow{d}{\eta^{L\dashv R}}\\
   (R\circ L)(A) \arrow[swap]{r}{R(f_\mathbf{D})} &(R\circ L)(B)
  \end{tikzcd}
 \end{equation}
 Then as the diagrams:
 \begin{equation*}
  \begin{tikzcd}
   (F\circ G)(A) \arrow[swap]{d}{\varepsilon^{F\dashv G}} \arrow{r}{F(G(f'))} &(F\circ G)(B) \arrow{d}{\varepsilon^{F\dashv G}}\\
   A \arrow{r}{f'} &B
  \end{tikzcd}
  \begin{tikzcd}
   A \arrow[swap]{d}{\eta^{L\dashv R}} \arrow{r}{f'} &B \arrow{d}{\eta^{L\dashv R}}\\
   (R\circ L)(A) \arrow[swap]{r}{R(L(f'))} &(R\circ L)(B)
  \end{tikzcd}
 \end{equation*}
 each also commute, we must have $f_\mathbf{B} = G(f')$ and $f_\mathbf{D} = L(f')$ by the universality of the counit and unit respectively.
 
 This means that the condition that every diagram of shape (\ref{fullness-unfilled}) can be filled in as (\ref{fullness-filled}) is equivalent to $\gamma$ being full.
 
 One easy way to achieve this is for $\varepsilon^{F\dashv G}$ to be epic at $A$, for $\eta^{L\dashv R}$ to be monic at $B$, and for one of those two to be strong.  The first two conditions would be guaranteed by faithfulness of $G$ and $L$ respectively.  Let's make that assumption.  Since $\mathbf{C}$ is complete and cocomplete, Proposition 4.3.7(3) of Borceux book 1 \cite{MR1291599} tells that $\varepsilon^{F\dashv G}$ will be strong epic as soon as $G$ is conservative, and that $\eta^{L\dashv R}$ will be strong monic as soon as $L$ is conservative.
 
 We conclude: 
 \begin{lemma}\label{full-faithful-hypothesis}
  If $G$ and $L$ are both faithful and at least one of the two is conservative, then $\gamma$ is full and faithful.
 \end{lemma}
 
 Note that this is not a necessary condition.  In the case of topological systems, $\mathcal{L}$ is not faithful, as can be seen by looking at maps from a nonempty space to an indiscrete space with more than one element.
 
 We settle here and turn our attention to the existence of a left adjoint to $\gamma$.
 
 Following the proof of Theorem 5.2.3 in Rydeheard and Burstall \cite{MR999925}, $(F\downarrow R)$ is complete because $R$ preserves limits and cocomplete because $F$ preserves colimits.  Those limits are constructed from limits in $\mathbf{B}$ and $\mathbf{D}$, so $\gamma$ preserves limits as soon as both $G$ and $L$ preserve limits themselves.
 
 $G$ is a right adjoint, so this is only a restriction on $L$.  We mention that in the (dual) case of topological systems, the forgetful functor $\mathcal{U}: \mathbf{Top} \to \mathbf{Set}$ has both a left adjoint and a right adjoint.
 
 By \Cref{full-faithful-hypothesis} and the special adjoint functor theorem, we conclude the following.
 
 \begin{theorem}\label{main-theorem}
  Let $\mathbf{B},\mathbf{C},\mathbf{D}$ be complete and cocomplete categories, where $\mathbf{C}$ admits a cogenerating set and is well-powered.  Let
  \begin{align*}
   F: \mathbf{B}&\rightleftarrows\mathbf{C} :G,\\
   L: \mathbf{C}&\rightleftarrows\mathbf{D} :S,R,
  \end{align*}
  with $F \dashv G$ and $S\dashv L \dashv R$.  Assume further that both $G$ and $L$ are faithful, and that at least one of the two is conservative.  Then $\mathbf{C}$ is a reflective subcategory of $(\mathbf{B}\downarrow(G\circ R)) \cong (F\downarrow R) \cong ((L\circ F)\downarrow\mathbf{D})$ embedded by:
  \begin{align*}
   \gamma &: \mathbf{C} \to (F\downarrow R),\\
   \gamma(A) &:= \eta^{L \dashv R}_A\circ\varepsilon^{F \dashv G}_A,\\
   \gamma(f) &:= (G(f)\downarrow L(f)),
  \end{align*}
  for $A$ and $f$ respectively an object and an arrow of $\mathbf{C}$.
 \end{theorem}
 
 Note that our hypotheses on $\mathbf{C}$ imply that $L$ preserving limits is equivalent to $L$ admitting a left adjoint $S$.
 
 \section{Simple graphs}\label{simple-graphs}
 
 The theorem requires an adjoint triple $S\dashv L\dashv R$.  One such triple is $\mathcal{D\dashv V\dashv I}$ between $\mathbf{Set}$ and the category $\mathbf{Gph}$ of simple graphs, where $\mathcal{V}: \mathbf{Gph} \to \mathbf{Set}$ is the functor taking a simple graph to its set of vertices, and $\mathcal{D,I}: \mathbf{Set} \to \mathbf{Gph}$ are the discrete and indiscrete graph functors respectively.
 
 Note that the category $\mathbf{Gph}$ here is the category of \emph{reflexive} graphs with no loops, where graph homomorphisms are permitted to collapse adjacent vertices to a single vertex.  This is following the nlab article \cite{nlab:category_of_simple_graphs}.  $\mathbf{Gph}$ is a Grothendieck quasitopos, which implies that it is locally presentable, which in turn (along with cocompleteness) implies that it is co-well-powerered, see \cite{nlab:well-powered_category}.
 
 Given a simple graph $G$, the right angled Artin group of $G$ is the quotient of the free group on the vertices of $G$ by the commutators of adjacent vertex pairs.  This gives a functor $\mathcal{A}: \mathbf{Gph} \to \mathbf{Grp}$.
 
 It is known (referenced in \cite{4029945}) that $\mathcal{A}$ has a right adjoint $\mathcal{C}: \mathbf{Grp} \to \mathbf{Gph}$, which takes a group $H$ to its commutation graph $\mathcal{C}(H)$.  The vertices of $\mathcal{C}(H)$ are the elements of $H$, where two vertices are defined to be adjacent iff they commute in $H$.
 
 \begin{corollary}
  $\mathbf{Gph}$ embeds as a full coreflective subcategory of $(\mathcal{F}\downarrow\mathbf{Grp})$, where $\mathcal{F}: \mathbf{Set} \to \mathbf{Grp}$ is the free group functor.
 \end{corollary}
 \begin{proof}
  The situation looks like:
 \begin{align*}
  \mathcal{D,I}: \mathbf{Set}&\rightleftarrows\mathbf{Gph} :\mathcal{V},\\
  \mathcal{A}: \mathbf{Gph}&\rightleftarrows\mathbf{Grp} :\mathcal{C},
 \end{align*}
 with $\mathcal{D \dashv V \dashv I}$ and $\mathcal{A \dashv C}$.  Here $\mathcal{V}$ and $\mathcal{A}$ are both faithful, and $\mathcal{A}$ is conservative.  The single-vertex graph is a generator of $\mathbf{Gph}$, and as noted earlier, the category is co-well-powered.
 
 Because $\mathcal{F = A\circ D}$, the result now follows directly from the dual to \Cref{main-theorem}.
 \end{proof}
 
 The embedding takes a graph to the quotient map from the free group on the vertices to the right angled Artin group.  A homomorphism out of a free group is isomorphic to a simple graph when it is surjective and its kernel is generated by commutators of generators.
 
 Our comma category $(\mathcal{F}\downarrow\mathbf{Grp})$ also contains the category $\mathbf{Grp}$ as a reflective subcategory, embedded by taking a group to its counit under the free-forgetful adjunction.  The left adjoint to this embedding is simply the codomain projection.
 
 \bibliography{bibliography}{}
\end{document}